\documentclass[12pt]{amsart}
\usepackage[utf8]{inputenc}
\usepackage{amsbsy}
\usepackage{amscd} 
\usepackage{amsfonts}
\usepackage{amsgen} 
\usepackage{amsmath}
\usepackage{amsopn} 
\usepackage{amssymb}
\usepackage{amstext}
\usepackage{amsthm} 
\usepackage{amsxtra}
\usepackage[all]{xy}
\usepackage{epsfig}
\usepackage{mathtools}
\usepackage{rotating}
 \usepackage{subfigure}

\include{FriezeCommands}

\theoremstyle{plain} 
\newtheorem{thm}{Theorem}[section]

\newtheorem{lem}[thm]{Lemma}

\theoremstyle{definition}
\newtheorem{defn}[thm]{Definition}
\newtheorem{rem}[thm]{Remark}
\newtheorem{ex}[thm]{Example}

\numberwithin{equation}{section}

\renewcommand{\theta}{\vartheta}
\renewcommand{\phi}{\varphi}
\renewcommand{\epsilon}{\varepsilon}

\newcommand{\N}{\mathbb N}
\newcommand{\Z}{\mathbb Z}

\DeclareMathOperator{\id}{id}

\begin{document}
\title{Factorization of Frieze Patterns}
\author{Moritz Weber and Mang Zhao}
\address{Saarland University, Fachbereich Mathematik, Postfach 151150,
66041 Saarbr\"ucken, Germany}
\email{weber@math.uni-sb.de, s8mazhao@stud.uni-saarland.de}
\date{\today}
\thanks{The first author was supported by the ERC Advanced Grant NCDFP, held by Roland Speicher, by the SFB-TRR 195, and by the DFG project \emph{Quantenautomorphismen von Graphen}. This article was part of the second author's Bachelor's thesis. We thank Michael Cuntz for introducing us to Frieze patterns and for discussions on the article.}

\begin{abstract}
In 2017, Michael Cuntz gave a definition of reducibility of quiddity cycles of frieze patterns: It is reducible if it can be written as a sum of two other quiddity cycles. We discuss the commutativity and associativity of this sum operator for quiddity cycles and its equivalence classes, respectively.
We show that the sum is neither commutative nor associative, but we may circumvent this issue by passing to equivalence classes. We also address the question whether a decomposition of quiddity cycles into irreducible factors is unique and we answer it in the negative by giving counterexamples. We conclude that even under stronger assumptions, there is no canonical decomposition.
\end{abstract}

\maketitle
\section*{Introduction}

Frieze patterns have been first introduced by H.S.M. Coxeter in 1971 \cite{FP}. They consist in rows of numbers, where the first two rows at top and bottom are 0's and 1's and the minor of every adjacent $2 \times 2$ entries is equal to one. Moreover, if the minor of every adjacent $3 \times 3$ entries is equal to zero, then this frieze pattern is tame. For example, a tame frieze pattern with width four is as follows:
\begin{table}[h]
	\centering
	\resizebox{\linewidth}{!}{
		\begin{tabular}{cccccccccccccccccccccccc}
			row 0 & 0 &  & 0 &  & 0 &  & 0 &  & 0 &  & 0 &  & 0 &  & 0 &  & 0 &  & 0 &  & 0 &  & 0 \\ 
			row 1 &  & 1 &  & 1 &  & 1 &  & 1 &  & 1 &  & 1 &  & 1 &  & 1 &  & 1 &  & 1 &  & 1 &  \\ 
			row 2 & ... &  & \textbf{3} &  & \textbf{1} &  & \textbf{2} &  & \textbf{4} &  & \textbf{1} &  & \textbf{2} &  & \textbf{2} &  & 3 &  & 1 &  & 2 &  & ... \\ 
			row 3 &  &  &  & 2 &  & 1 &  & 7 &  & 3 &  & 1 &  & 3 &  & 5 &  & 2 &  & 1 &  & ... &  \\ 
			row 4 & ... &  & 3 &  & 1 &  & 3 &  & 5 &  & 2 &  & 1 &  & 7 &  & 3 &  & 1 &  & 3 &  & ... \\ 
			row 5 &  &  &  & 1 &  & 2 &  & 2 &  & 3 &  & 1 &  & 2 &  & 4 &  & 1 &  & 2 &  & ... &  \\ 
			row 6 & ... &  & 1 &  & 1 &  & 1 &  & 1 &  & 1 &  & 1 &  & 1 &  & 1 &  & 1 &  & 1 &  & ... \\ 
			row 7 &  & 0 &  & 0 &  & 0 &  & 0 &  & 0 &  & 0 &  & 0 &  & 0 &  & 0 &  & 0 &  & 0 &  \\ 
		\end{tabular} 
	}
\end{table}

Note that each tame frieze pattern follows a certain periodicity in its entries. Now, a quiddity cycle is a finite sequence, which consists of elements in a period from row 2 \cite{FP} see the bold face sequence above. Using quiddity cycles, it becomes possible to investigate properties of a 2-dimensional frieze pattern in terms of properties of its 1-dimensional quiddity cycles. Recent results about frieze patterns, the extension to $SL_{k+1}$-friezes over integers, the close connection to geometry and algebra, as well as the combinatorial interpretations of integer-valued frieze patterns have been summarized in \cite{CFPATCOAGAC}.

In 2017, a new approach to build tame frieze patterns by using quiddity cycles has been introduced by Michael Cuntz and Thorsten Holm \cite{FPOIAOSOTCN}. Meanwhile, the domain of elements in a frieze pattern has been widened to a subset of any commutative ring of numbers \cite{FPOIAOSOTCN,ACMFRFP}. Moreover, Cuntz gave a definition for reducing a quiddity cycle into two other quiddity cycles with respect to some "$\oplus$" operator \cite{ACMFRFP}. We ask the following natural questions:
\begin{enumerate}
	\item Is the "$\oplus$" operator commutative?
	\item Is the "$\oplus$" operator associative?
	\item Obviously, each quiddity cycle can be decomposed until all of its factors are irreducible. Is it unique? Does a general decomposition form exist?
\end{enumerate}

The answer to all these questions is no, as we will show in this article (Example 2.1 and Section 3.2). However, $(1)$ and $(2)$ have positive answers if we pass to equivalence classes (Theorem 2.3).

\section{Quiddity Cycles and the $\oplus$ Operator}

\subsection{Quiddity cycles}
The following definition has been given by Cuntz in 2017 \cite{ACMFRFP}, see also \cite{FPOIAOSOTCN}. Through out this article, $R$ denotes a subset of a commutative ring and $\lambda$ is a parameter $\lambda \in \{\pm 1\}$.
\begin{defn}\label{DefCun}\cite{ACMFRFP} A \underline{\textit{$\lambda$-quiddity cycle}} over $R$ is a finite sequence $c = (c_{1}, c_{2}, ..., c_{n}) \in R^{n}$ satisfying:
	$$ \prod_{k=1}^{n} \left( \begin{array}{cc}
	c_{k} & -1 \\
	1 & 0
	\end{array} \right) =  \left( \begin{array}{cc}
	\lambda & 0 \\
	0 & \lambda
	\end{array} \right) = \lambda \id $$
\end{defn}

\begin{rem}
	Given a $(-1)$-quiddity cycle, one can associate a tame frieze pattern to it, see \cite{FPOIAOSOTCN} [Prop. 2.4].
\end{rem}

\begin{ex}\label{ExCun}	\cite{ACMFRFP} In the case $R=\mathbb{C}$, we have the following results:
	\begin{description}
		\item[$(1)$] $(0,0)$ is the only $\lambda$-quiddity cycle of length 2.
		\item[$(2)$] $(-1,-1,-1)$ is the only $1$-quiddity cycle of length 3.
		\item[$(3)$] $(1,1,1)$ is the only $(-1)$-quiddity cycle of length 3.
		\item[$(4)$] $(t,\frac{2}{t},t,\frac{2}{t})$, $t$ a unit and $(a,0,-a,0)$, $a$ arbitrary, are the only $\lambda$-quiddity cycles of length 4.
	\end{description}
\end{ex}

\begin{ex}
	\begin{enumerate}
		\item $c = (4,0,-3,-1,0,2,1) \in R^{7}$ is a $(-1)$-quiddity cycle over $R = \Z$.
		\item $c = (3,1,2,4,1,2,2)\in R^{7}$ is a $(-1)$-quiddity cycle over $R = \Z$.
	\end{enumerate}
\end{ex}

\subsection{Equivalence classes of quiddity cycles}
\begin{defn}
	Let $c = (c_{1},...,c_{n}) \in R^{n}$ be a $\lambda$-quiddity cycle and let $\sigma \in D_{n}$, where $D_{n}$ is the dihedral group acting on $n$ points. We put 
	$$c^{\sigma} := (c_{\sigma_{1}},...,c_{\sigma_{n}}).$$
	Moreover, we define the following equivalence relation for tuples $a,b \in R^{n}$.
	$$ a \sim b :\Leftrightarrow \exists \sigma \in D_{n}: a = b^{\sigma}$$
	For convenience, we put $D_{\infty}:= \bigcup_{n \in \N} D_{n}$ and we write $c^{\sigma}$ for $\sigma \in D_{\infty}$ only if $\sigma$ has an appropriate length.
\end{defn}

\begin{rem}\label{QCDihedral}
	The sequence $c^{\sigma}$ is a $\lambda$-quiddity cycle as well \cite{FPOIAOSOTCN} [Prop 2.6]. Note that by definition $c^{\sigma}$ is a cyclic rotation of $(c_{1},c_{2},...,c_{n})$ or of $(c_{n},...,c_{2},c_{1})$.
\end{rem}

\begin{lem}\label{EquiRelation} For $n\in \N$ and $a,b \in R^{n}, a\sim b$ defines an equivalence relation on $R$.
\end{lem}
\begin{proof}
	Given $a=b^{\sigma}$ and $b=c^{\pi}$, we have $a = a^{\id}$, $b = a^{\sigma^{-1}}$ and $a = c^{\sigma \pi}$ which proves reflexivity, symmetry and transitivity.
\end{proof}

\subsection{The $\oplus$ operator}
\begin{defn}\cite{ACMFRFP}
	Let $a = (a_{1}, ..., a_{k})$ be a $\lambda$-quiddity cycle and $b = (b_{1}, ..., b_{l})$ be a $\lambda^{\prime}$-quiddity cycle. We put
	 \[a \oplus b = (a_{1} + b_{l}, a_{2},...,a_{k-1},a_{k}+b_{1},b_{2},...,b_{l-1}).\]
\end{defn}

\begin{lem}\label{RelationSumQC} Let  $a = (a_{1}, a_{2}, ..., a_{k})$ be a $\lambda^{ \prime}$-quiddity cycle and $b = (b_{1}, b_{2}, ..., b_{l}) \in R^{l}$ be any finite sequence. Then $b$ is a $\lambda^{\prime\prime}$-quiddity cycle if and only if $ a \oplus b  $ is a $(- \lambda^{ \prime}\lambda^{\prime\prime})$-quiddity cycle.
\end{lem}

\begin{proof} \allowdisplaybreaks
If $a$ and $b$ are quiddity cycles, so is $a \oplus b$ by \cite{ACMFRFP} [Lemma 2.7]. Conversely, let $a \oplus b$ be a $(- \lambda^{ \prime}\lambda^{\prime\prime})$-quiddity cycle. Then
\begin{align*}
	&\quad (- \lambda^{ \prime}\lambda^{\prime\prime}) \id\\
	&=\left(\begin{array}{cc}
	a_{1} + b_{l} & -1\\
	1 & 0
	\end{array} \right)   \prod_{t=2}^{k-1}\left(\begin{array}{cc}
	a_{t} & -1\\
	1 & 0
	\end{array} \right) \left( \begin{array}{cc}
	a_{k} + b_{1} & -1\\
	1 & 0
	\end{array}\right) \prod_{t=2}^{l-1}\left(\begin{array}{cc}
	b_{t} & -1\\
	1 & 0
	\end{array} \right) \\
	&= \left(\begin{array}{cc}
	b_{l} & -1\\
	1 & 0
	\end{array} \right)\left(\begin{array}{cc}
	0 & -1\\
	1 & 0
	\end{array} \right)\left(\begin{array}{cc}
	a_{1} & -1\\
	1 & 0
	\end{array} \right)\prod_{t=2}^{k-1}\left(\begin{array}{cc}
	a_{t} & -1\\
	1 & 0
	\end{array} \right) \left( \begin{array}{cc}
	a_{k} & -1\\
	1 & 0
	\end{array}\right)\\
	&\quad \left( \begin{array}{cc}
	0 & -1\\
	1 & 0
	\end{array}\right)\left( \begin{array}{cc}
	b_{1} & -1\\
	1 & 0
	\end{array}\right) \prod_{t=2}^{l-1}\left(\begin{array}{cc}
	b_{t} & -1\\
	1 & 0
	\end{array} \right)\\
	&= (- \lambda^{ \prime})\left(\begin{array}{cc}
	b_{l} & -1\\
	1 & 0
	\end{array} \right) \prod_{t=1}^{l-1}\left(\begin{array}{cc}
	b_{t} & -1\\
	1 & 0
	\end{array} \right)
\end{align*} 
	Therefore, $(b_{l},b_{1},...,b_{l-1})$ is a $\lambda^{\prime\prime}$-quiddity cycle. With Remark \ref{QCDihedral} we know that $b$ is also a $\lambda^{\prime\prime}$-quiddity cycle. 
\end{proof}

\section{Commutativity and Associativity for $\oplus$}

We now show that the sum of quiddity cycles is neither commutative nor associative in general. However, passing to equivalence classes solves this issue.

\begin{ex}\label{RemNonCommAsso} 
	\begin{enumerate}
		\item We do not have $a \oplus b = b \oplus a$ in general. For example, if $a = (1,1,1)$ and $b = (2,1,2,1)$, then $$a \oplus b = (2, 1, 3, 1, 2) \neq (3, 1, 2, 2, 1) = b \oplus a.$$
		Note that we have $a \oplus b \sim b \oplus a$.
		
		\item We do not have $(a \oplus b) \oplus c = a \oplus (b \oplus c)$ in general. For example, if $a = (1,1,1)$, $b = (2,1,2,1)$ and $c = (1,1,1)$, then 
		\[\begin{aligned}
			&(a \oplus b) \oplus c&&= (2, 1, 3, 1, 2) \oplus (1,1,1) \\
			&\quad &&= (3,1,3,1,3,1) \\
			&\quad &&\neq (2,1,4,1,2,2)\\
			&\quad &&= (1,1,1) \oplus (3, 1, 2, 2, 1)\\
			&\quad &&= a \oplus (b \oplus c)
		\end{aligned}\] 
		
		Note that we even have $(a \oplus b) \oplus c \nsim a \oplus (b \oplus c)$.
		
		\item If $a \oplus b \sim a \oplus c$, then we may have $b \neq c$ and even $b \nsim c$. For example, \[(1,1,1) \oplus (0,6,0,-6) = (-5,1,1,6,0)\]
		\[	(1,1,1) \oplus (5,0,-5,0) = (1,1,6,0,-5)\]
	\end{enumerate}
\end{ex}

Hence, the sum does not behave nicely on the level of quiddity cycles. However, on equivalence classes the situation is better. Before proving the main theorem of this article, we need to prove some technical lemma.

\begin{lem}\label{abExchange}
	Let $a = (a_{1},...,a_{k})$ be a $\lambda$-quiddity cycle and let $b = (b_{1},...,b_{l})$ be a $\lambda^{\prime}$-quiddity cycle. Let $\widetilde{a} = (a_{k},...,a_{1})$ and $\widetilde{b} = (b_{l},...,b_{1})$. Then $a \oplus b \sim \widetilde{a} \oplus \widetilde{b}$.
\end{lem}
\begin{proof}
	We have $$ a \oplus b = (a_{1} + b_{l}, a_{2},...,a_{k-1},a_{k}+b_{1},b_{2},...,b_{l-1}) $$
	and \[\widetilde{a} \oplus \widetilde{b} = (a_{k}+b_{1},a_{k-1},...,a_{2}, a_{1} + b_{l}, b_{l-1},...,b_{2}).\]
\end{proof}

\begin{thm}\label{CommutativityAssociativity} Let $k,l,m>2$ and let 
\begin{itemize}
	\item $a = (a_{1},a_{2}, ..., a_{k}) \in R^{k}$ be a $\lambda$-quiddity cycle
	\item $b = (b_{1}, b_{2}, ..., b_{l}) \in R^{l}$ be a $\lambda^{\prime}$-quiddity cycle
	\item $c = (c_{1}, c_{2}, ..., c_{m}) \in R^{m}$ be a $\lambda^{\prime\prime}$-quiddity cycle.
\end{itemize}
Then we have:
\begin{enumerate}
	\item $a \oplus b \sim b \oplus a$.
	
	\item $\exists \sigma,\tau \in D_{\infty}: (a \oplus b) \oplus c = a \oplus (b^{\tau} \oplus c)^{\sigma}$.
	
	\item Let $\sigma \in D_{\infty}$. Then either $$a\oplus (b \oplus c)^{\sigma} \sim (a \oplus c^{\tau_{1}})^{\tau_{2}} \oplus b$$ or $$a \oplus (b \oplus c)^{\sigma} \sim (a \oplus b^{\tau_{1}})^{\tau_{2}} \oplus c$$ for some $\tau_{1},\tau_{2} \in D_{\infty}$.
	
	\item If $k = l$, $\lambda = \lambda^{\prime}$ and $a \oplus b = b \oplus a$, then $a=b$.
\end{enumerate}
\end{thm}
\begin{proof}
	\begin{enumerate}
		\item Let $\sigma \in D_{\infty}$ be given by $(1,...,k+l-2) \mapsto (l,...,k+l-2,1,...,l-1)$. Then:
		\[\begin{aligned}
		&a \oplus b &&= (a_{1} + b_{l}, a_{2}, ..., a_{k-1}, a_{k} + b_{1}, b_{2}, ..., b_{l-1})\\
		&\quad &&= (a_{k} + b_{1}, b_{2}, ..., b_{l-1}, a_{1} + b_{l}, a_{2}, ..., a_{k-1})^{\sigma}\\
		&\quad &&= (b \oplus a)^{\sigma}
		\end{aligned}\]

		\item Let  $\sigma \in D_{\infty}$ be given by $(1,...,l+m-2) \mapsto (2,...,l+m-2,1)$ and $\tau \in D_{\infty}$ be given by $(1,...,l) \mapsto (l,1,...,l-1)$. Then: 		
		 \[\begin{aligned}
		&(a \oplus b) \oplus c &&= (a_{1} + b_{l}, a_{2}, ..., a_{k-1}, a_{k} + b_{1}, b_{2}, ..., b_{l-1}) \oplus c \\
		&\quad &&= (a_{1} + b_{l} + c_{m}, a_{2}, ..., a_{k-1}, a_{k} + b_{1}, b_{2}, ..., b_{l-2}, b_{l-1}+c_{1}, c_{2}, ..., c_{m-1})\\
		&\quad &&= a \oplus (b_{1}, b_{2}, ..., b_{l-2}, b_{l-1}+c_{1}, c_{2}, ..., c_{m-1},b_{l} + c_{m}) \\
		&\quad &&= a \oplus (b_{l} + c_{m}, b_{1}, b_{2}, ..., b_{l-2}, b_{l-1}+c_{1}, c_{2}, ..., c_{m-1})^{\sigma}\\ 
		&\quad &&= a \oplus ((b_{l}, b_{1}, b_{2}, ..., b_{l-1}) \oplus c)^{\sigma} \\
		&\quad &&= a \oplus (b^{\tau} \oplus c)^{\sigma}
		\end{aligned}\] 
		
		\item We split the proof of $(3)$ into four cases depending on $\sigma$. Note that 
		$$b \oplus c = (b_{1}+c_{m},b_{2},...,b_{l-1},b_{l}+c_{1},c_{2},...,c_{m-1}).$$
		\begin{description}
			\item[(3.1)] If $(b \oplus c)^{\sigma} = b \oplus c$, then using $(1)$ and $(2)$, we find $\tau_{1},\tau_{2},\pi \in D_{\infty}$ such that: $$a \oplus (b \oplus c)^{\sigma} \sim (b \oplus c) \oplus a = b \oplus (c^{\tau_{1}} \oplus a)^{\pi} \sim (c^{\tau_{1}} \oplus a)^{\pi} \oplus b = (a \oplus c^{\tau_{1}})^{\tau_{2}} \oplus b $$
			
			\item[(3.2)] If $(b \oplus c)^{\sigma} = (b_{i}, b_{i+1} ..., b_{l-1}, b_{l} + c_{1}, c_{2}, ..., c_{m-1}, b_{1} + c_{m}, b_{2},...,b_{i-1})$ for some $2 \leq i \leq l-1$, using $(1)$ we find $\tau_{1}, \tau_{2} \in D_{\infty}$ such that:
			\begin{align*}		
			&\quad a \oplus (b \oplus c)^{\sigma}\\
			&=(a_{1} + b_{i-1}, a_{2},...,a_{k-1}, a_{k} + b_{i}, b_{i+1} ..., b_{l-1}, b_{l} + c_{1}, c_{2}, ..., c_{m-1}, b_{1} + c_{m}, b_{2},...,\\
			&\quad b_{i-2})\\
			&\sim(b_{l} + c_{1}, c_{2}, ..., c_{m-1}, b_{1} + c_{m}, b_{2},...,b_{i-2},a_{1} + b_{i-1}, a_{2},...,a_{k-1}, a_{k} + b_{i}, b_{i+1} ...,\\
			&\quad b_{l-1})\\
			&=c \oplus (b_{1}, b_{2},...,b_{i-2},a_{1} + b_{i-1}, a_{2},...,a_{k-1}, a_{k} + b_{i}, b_{i+1} ..., b_{l-1}, b_{l})\\	
			&=c \oplus (a \oplus b^{\tau_{1}})^{\tau_{2}}\\	
			&\sim(a \oplus b^{\tau_{1}})^{\tau_{2}} \oplus c
			\end{align*}
			
			\item[(3.3)] If $\sigma$ is of the form $(1,...,l+m-2) \mapsto (i,i+1,...,l+m-2,1,...,i-1)$ for some $l \leq i \leq l+m-2$, we find some $\sigma^{\prime} \in D_{\infty}$ such that $a \oplus (b \oplus c)^{\sigma} = a \oplus (c \oplus b)^{\sigma^{\prime}}$ and apply the results from cases (3.1) and (3.2).
			
			\item[(3.4)] If $\sigma$ is of the form $(1,...,l+m+2) \mapsto (i,i-1,...,1,l+m-2,...,i+1)$ for some $1 \leq i \leq l+m-2$, we consider $\widetilde{b} = (b_{l},...,b_{2},b_{1})$ and $\widetilde{c} = (c_{m},...,c_{2},c_{1})$. By using Lemma \ref{abExchange}, we find $\widetilde{\sigma} \in D_{\infty}$ such that $(b \oplus c)^{\sigma} = (\widetilde{c} \oplus \widetilde{b})^{\widetilde{\sigma}}$ and we may use the results from the cases $(3.1),(3.2)$ and $(3.3)$ to conclude the proof with the help of Lemma \ref{abExchange}.
		\end{description}
		
		\item As $k = l$, we have
		$$ a \oplus b = (a_{1} + b_{k}, a_{2}, ..., a_{k-1}, a_{k} + b_{1}, b_{2}, ..., b_{k-1})$$ and $$b \oplus a = (b_{1} + a_{k}, b_{2}, ..., b_{k-1}, b_{k} + a_{1}, a_{2}, ..., a_{k-1}).$$ If $a \oplus b = b \oplus a$, we thus obtain $a_{t} = b_{t}$ for $2 \leq t \leq k-1$. Put 
		$$ \left( \begin{array}{cc}
		m_{1,1} & m_{1,2} \\
		m_{2,1} & m_{2,2}
		\end{array} \right) := \prod_{t=2}^{k-1} \left( \begin{array}{cc}
		a_{t} & -1 \\
		1 & 0
		\end{array} \right)= \prod_{t=2}^{k-1} \left( \begin{array}{cc}
		b_{t} & -1 \\
		1 & 0
		\end{array} \right).$$
		We then have
		\[
		\begin{aligned}
		\lambda \id &= \prod_{t=1}^{k} \left( \begin{array}{cc}
		a_{t} & -1 \\
		1 & 0
		\end{array} \right) = \left( \begin{array}{cc}
		a_{1} & -1 \\
		1 & 0
		\end{array} \right)\left( \begin{array}{cc}
		m_{1,1} & m_{1,2} \\
		m_{2,1} & m_{2,2}
		\end{array} \right)\left( \begin{array}{cc}
		a_{k} & -1 \\
		1 & 0
		\end{array} \right)\\
		&= \left( \begin{array}{cc}
		m_{1,1}a_{1}a_{k} - m_{2,1}a_{k} + m_{1,2}a_{1} - m_{2,2} & m_{2,1}-m_{1,1}a_{1} \\
		m_{1,1}a_{k} + m_{1,2} & -m_{1,1}
		\end{array} \right)
		\end{aligned} \]
		
		and as $\lambda = \lambda^{\prime}$ also
		\[
		\begin{aligned}
		\lambda\id &= \prod_{t=1}^{k} \left( \begin{array}{cc}
		b_{t} & -1 \\
		1 & 0
		\end{array} \right) = \left( \begin{array}{cc}
		b_{1} & -1 \\
		1 & 0
		\end{array} \right)\left( \begin{array}{cc}
		m_{1,1} & m_{1,2} \\
		m_{2,1} & m_{2,2}
		\end{array} \right)\left( \begin{array}{cc}
		b_{k} & -1 \\
		1 & 0
		\end{array} \right)\\
		&= \left( \begin{array}{cc}
		m_{1,1}b_{1}b_{k} - m_{2,1}b_{k} + m_{1,2}b_{1} - m_{2,2} & m_{2,1}-m_{1,1}b_{1} \\
		m_{1,1}b_{k} + m_{1,2} & -m_{1,1}
		\end{array} \right)
		\end{aligned}
		 \]
		As $m_{1,1} = -\lambda \neq 0$, this implies $a_{1} = b_{1}$ and $a_{k} = b_{k}$ and therefore $a=b$.
	\end{enumerate}
\end{proof}

\section{Factorization and reducibility}

In \cite{ACMFRFP} Cuntz introduced the concept of irreducibility for quiddity cycles. In this section, we address the question whether the decomposition into irreducible factors is unique. We give counterexamples and remark that there cannot be a canonical decomposition.

\subsection{Reducibility}
\begin{defn}\label{Reducibility}
\cite{ACMFRFP} Let $m > 2$ and let $c = (c_{1}, c_{2}, ..., c_{m}) \in R^{m}$ be a $\lambda$-quiddity cycle. Then, $c$ is called \underline{\textit{reducible}} over $R$ if there exists a $\lambda^{\prime}$-quiddity cycle $a = (a_{1}, a_{2}, ..., a_{k})$ and a $\lambda^{\prime\prime}$-quiddity cycle $b = (b_{1}, b_{2}, ..., b_{l})$, such that:
	\begin{enumerate}
		\item $\lambda = - \lambda^{\prime}\lambda^{\prime\prime}$
		\item $k,l>2$ and $m=k+l-2$
		\item $c \sim a \oplus b$
	\end{enumerate}
If $c$ is not reducible then $c$ is called \underline{\textit{irreducible}}.
\end{defn}

\begin{ex}\label{ExReducibility}\cite{ACMFRFP}
	\begin{enumerate}
		\item The set of irreducible $\lambda$-quiddity cycles over  $\mathbb{Z_{\geq \text{0}}}$ is $\{(0,0,0,0),(1,1,1)\}$.
		\item The set of irreducible $\lambda$-quiddity cycles over  $\mathbb{Z}$ is \[\{(1,1,1),(-1,-1,-1),(a,0,-a,0),(0,a,0,-a) | a\in \mathbb{Z} \backslash \{0,\pm 1\} \}.\]
		\item $(0,-1,0,1)$ is reducible over $\mathbb{Z}$, since $(0,-1,0,1) = (-1,-1,-1) \oplus (1,1,1)$.
	\end{enumerate}
\end{ex}

\begin{ex}\label{ReducibilityRing}
	\begin{enumerate}
		\item All $\lambda$-quiddity cycles having the form $(1,...,1)$ are irreducible over $R=\N_{>0}$.
		\begin{proof}
			Suppose $c=(c_{1},...,c_{m})=(1,...,1)\in R^{m}$ is reducible as $c = (a \oplus b)^{\sigma}$ with $a\in R^{k}, b\in R^{l}, \sigma \in D_{\infty}$. Then we obtain $m = k+l-2$ and $m = \sum_{t=1}^{m}c_{t} = \sum_{t=1}^{k}a_{t}+\sum_{t=1}^{l}b_{t} \ge k+l = m+2$, which is a contradiction.
		\end{proof}
	
		\item The reducibility of a $\lambda$-quiddity cycle is dependent on the set $R$. For example, $(1,1,1,1,1,1,1,1,1)$ is irreducible over $R=\N_{>0}$ by $(1)$. \\
		But $(1,1,1,1,1,1,1,1,1) = (1,1,1) \oplus (0,1,1,1,1,1,1,0)$ is reducible over $R=\mathbb{N}$.
	\end{enumerate}
\end{ex}

\subsection{The decomposition into irreducible factors is not unique}
By Section 2 we know that $\oplus$ is not associative. Hence, a decomposition into irreducible factors seems to be problematic. Indeed consider the $(-1)$-quiddity cycle $c = (4,0,-3,-1,0,2,1) \in \Z^{7}$ with $R = \Z$. From Example \ref{ExReducibility} we know all irreducible $(-1)$-quiddity cycles over $\Z$. We may decompose $c$ as follows:

\begin{figure}[h]
	\centering
	\includegraphics[width=0.7\linewidth]{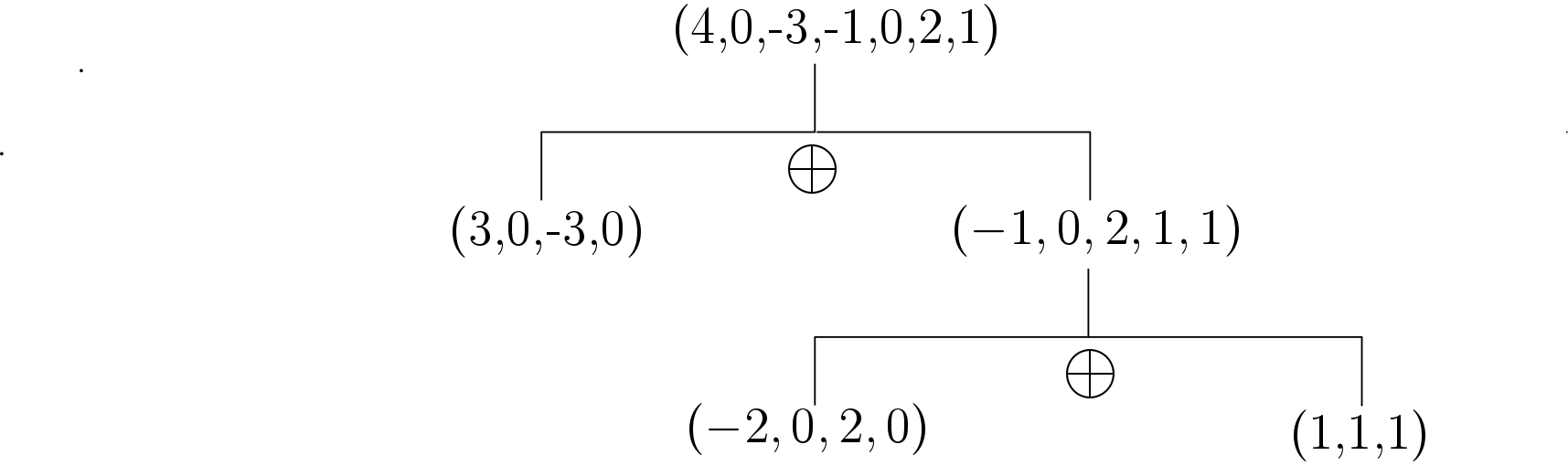}
	\label{fig:red}
\end{figure}

\pagebreak
However, we may also choose the following decomposition:

\begin{figure}[h]
	\centering
	\includegraphics[width=1.0\linewidth]{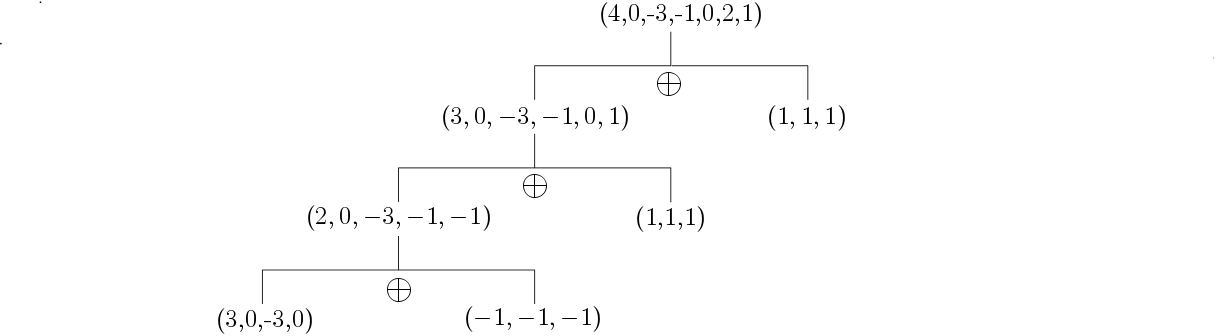}
	\label{fig:red1}
\end{figure}

We observe that different irreducible factors may appear and even the number of irreducible factors may vary.

\subsection{There is no canonical decomposition into irreducible factors}
In Section 3.2, we saw an example of a $(-1)$-quiddity cycle c over $\Z$ with decompositions $c = a_{1} \oplus (a_{2} \oplus a_{3})$ and $c = ((b_{1} \oplus b_{2}) \oplus b_{3}) \oplus b_{4}$ into irreducible factors.\\
So, as the decomposition into irreducible factors is not unique, we might wonder whether we may always find a unique factorization of the form: $$c = (((a_{1} \oplus a_{2}) \oplus a_{3}) \oplus ...) \oplus a_{n}$$
But even this is not true.\\
For instance, consider the $1$-quiddity cycle $c = (2,1,1,-1,0) \in \Z^{5}$. Then there is no irreducible quiddity cycle $b$ such that $c = a \oplus b$.\\
Indeed, if there was such a quiddity cycle $b = (b_{1},...,b_{l})$, the last entry of $c$ would be the entry $b_{l-1}$. The only irreducible quiddity cycle with $b_{l-1} = 0$ is $b = (0, \lambda, 0, -\lambda)$ for some $\lambda \in \Z \backslash \{-1,0,1\}$.\\
Hence $a$ must be of length three and we have
$$ (2,1,1,-1,0) = (a_{1} - \lambda, a_{2},a_{3}+0,\lambda,0)$$
which implies $a = (1,1,1)$ and $b = (0,-1,0,1)$, but $b$ is not irreducible by Example \ref{ExReducibility}(3).\\
We conclude that we cannot avoid the action of the dihedral group in a factorization procedure. Using Theorem \ref{CommutativityAssociativity}, one can prove that for a given $\lambda$-quiddity cycle $c \in R_{m}$, we may always find $\sigma_{1},...,\sigma_{n-1} \in D_{\infty}$ such that $$c = ((((a_{1} \oplus a_{2})^{\sigma_{1}} \oplus a_{3})^{\sigma_{2}} \oplus ... )^{\sigma_{n-2}} \oplus a_{n})^{\sigma_{n-1}}.$$
However, we may not choose $\sigma_{1} = ... = \sigma_{n-1} = \id$ in general (see the above example) nor is the above factorization unique in any sense, see the example in Section 3.2 with $c = ((b_{1} \oplus b_{2}) \oplus b_{3}) \oplus b_{4}$ and $c = ((a_{2} \oplus a_{3}) \oplus a_{1} )^{\sigma}$, using Theorem 
\ref{CommutativityAssociativity}.

\bibliographystyle{alpha}

\bibliography{FriezeBib}

\end{document}